\documentclass{amsart}
\usepackage{amssymb,amsmath, amsthm,latexsym}
\usepackage{graphics}
\usepackage{psfrag}
\usepackage{amscd} 
\usepackage{graphicx}
\newcommand{\cal}[1]{\mathcal{#1}}
\theoremstyle{plain}
\newtheorem{theo}{Theorem}
 
\newtheorem{lemma}{Lemma}[section]

\newtheorem{proposition}[lemma]{Proposition}
\newtheorem{corollary}[lemma]{Corollary}
\theoremstyle{definition}
\newtheorem{definition}[lemma]{Definition}
\newtheorem{remark}[lemma]{Remark} 
\newtheorem{example}[lemma]{Example}

\parskip=\bigskipamount

\let\egthree=\phi 
\let\phi=\varphi
\let\varphi=\egthree

\newcounter{sebcomments}
\begin{document}
\title{Dynamical properties of the absolute period foliation}
\author{Ursula Hamenst\"adt}
\thanks
{
Research
partially supported by ERC grant 10160104}
\date{November 24, 2015} 

\begin{abstract} We show that the absolute period foliation
of the principal stratum of abelian differentials on a surface
of genus $g\geq 3$ is ergodic.  We also investigate the absolute
period foliation on affine invariant manifolds.
\end{abstract}

\maketitle

\section{Introduction}

The moduli space of abelian differentials on a surface
of genus $g\geq 2$ naturally decomposes into \emph{strata}
of differentials with prescribed numbers of zeros. 
Period coordinates on such a stratum ${\cal Q}$ are defined
by evaluation of an abelian 
differential on a basis for relative homology.
If ${\cal Q}$ is a stratum of differentials with 
more than one zero then it admits a 
natural foliation whose leaves consist 
of differentials with (locally) fixed absolute
periods. This foliation is smooth and has been analyzed in 
\cite{McM13,McM14,MinW14}; it is called the \emph{absolute period
foliation}. 

A smooth foliation of an orbifold ${\cal Q}$ is called \emph{ergodic} 
if any Borel subset of ${\cal Q}$ which is saturated for the foliation
either has full or vanishing Lebesgue measure. In \cite{McM14},
tools from homogeneous dynamics are used to 
show that the absolute period foliation of the principal stratum in 
$g=2,3$ is ergodic. 
Calsamiglia, Deroin and Francaviglia 
\cite{CDF15} completely classified the closures of the
leaves of the absolute period foliation. As a consequence,
they obtain ergodicity of the absolute period foliation 
of the principal stratum in every genus.

Our main goal is to give a simple proof of the latter fact.

\begin{theo}\label{ergodicmain}
The absolute period foliation of the principal stratum is ergodic in every 
genus $g\geq 2$.
\end{theo}

We do not know whether the absolute period foliation 
of a stratum which is not principal is ergodic. 
Using an argument of Coud\'ene \cite{C09}, it is not hard to 
see that ergodicity is equivalent to the existence of a dense leaf.

By the groundbreaking work of Eskin, Mirzakhani and Mohammadi
\cite{EMM13}, the closure of an orbit 
for the $SL(2,\mathbb{R})$-action on any stratum ${\cal Q}$ 
is an affine invariant manifold. 
Examples of non-trivial orbit closures are arithmetic
Teichm\"uller curves. They arise from branched covers
of the torus, and they 
are dense in any stratum of abelian differentials.
Other examples of orbit closures different from entire
components of strata can be constructed using more general
branched coverings.



In Section 3 we
introduce rigid and flexible tangent fields of the absolute
period foliation and use this to investigate
the principal boundary of an affine invariant manifold.

\noindent
{\bf Acknowledgement:} I am grateful to Barak Weiss for 
useful discussions.

\section{The absolute period foliation}\label{absoluteperiod}

Let ${\cal Q}$  be a component of a stratum 
with $k\geq 2$ zeros
in the moduli space of abelian differentials on a surface
of genus $g\geq 2$.
The absolute periods of 
an abelian differential $\omega\in {\cal Q}$
define a local submersion of orbifolds
\[{\cal Q}\supset U\to H^1(X,\mathbb{C})/{\rm Aut}(X)\]
whose fibres are the intersections with ${\cal Q}$ 
of the leaves of the 
\emph{absolute period foliation} ${\cal A\cal P}({\cal Q})$.
This foliation is transverse to the fibres of the canonical
projection $\pi:{\cal Q}\to {\cal M}_g$ (here ${\cal M}_g$ denotes
the moduli space of Riemann surfaces of genus $g$)
and to the orbits of the natural action of $SL(2,\mathbb{R})$.

The leaf ${\cal A\cal P}(\omega)$ 
of ${\cal A\cal P}({\cal Q})$ through $\omega$
is locally a flat submanifold
of ${\cal Q}$ which can explicitly be 
described \cite{MinW14,McM13}.

Assume for the moment that ${\cal Q}$ is the principal stratum.
Denote by $Z(\omega)$ the zero set of $\omega\in {\cal Q}$.
The cardinality of $Z(\omega)$ equals $2g-2$.
At each zero 
$p\in Z(\omega)$ there is an infinitesimal 
deformation of $\omega$ called the \emph{Schiffer
variation} \cite{McM13} which is defined as follows.

Let $X$ be the Riemann surface underlying $\omega$. 
Choose a complex coordinate $z$ for $X$ near $p$ so that 
in this coordinate, $\omega$ can be written as
$\omega=(z/2)dz$. Such a coordinate is unique up to multiplication
with $-1$. 
Choose a vertical arc $A_t=i[-2u,2u]$ in this coordinate 
where $t=u^2$. Slit $S$ open along $A_t$ and fold 
each of the two resulting
arcs so that $z$ is identified with $-z$ (see p.1235 
of \cite{McM13}).

The result is a new Riemann surface $X_t$ with a 
distinguished horizontal arc $B_t$
and a natural holomorphic map $f_t:X-A_t\to X-B_t$. 
The one-form $\omega_t$ with $f_t^*\omega_t=\omega$ is globally defined, and
it only depends on the parameter $t$ and 
on the choice of the zero $p$ of $\omega$. 
The Schiffer variation of $X$ is 
\[{\rm Sch}(\omega,p)=dX_t/dt\vert_{t=0}.\]

It will be useful to have a geometric description 
of the deformation of the one-forms $\omega_t$ defining the  
Schiffer variation. Namely, 
there are four horizontal separatrices at $p$ for the flat metric 
defined by $\omega$. In a complex coordinate $z$ near  
$p$ so that $\omega=(z/2)dz$, 
the horizontal separatrices are the four rays contained in the 
real or the imaginary axis. 
The restriction of $\omega$ to these rays defines
an orientation on the rays. With respect to this orientation, 
the two rays contained in the real
axis are outgoing from $p$, 
while the rays contained in the
imaginary axis are incoming. The Schiffer variation slides the singular point 
backwards along the incoming rays in the 
imaginary axis. 
Thus if one of the two separatrices in the
imaginary axis is a saddle connection for $\omega$, then the flat length 
of the
corresponding saddle connection for $\omega_t$ is decreasing with $t$.


If $\omega$ has a zero of order $n\geq 2$ at $p$ then 
the Schiffer variation at $p$ is defined as follows (see
p.1235 of \cite{McM13}).  Choose a coordinate $z$ near
$p$ so that $\omega=z^ndz$ in this coordinate. 
This choice of coordinate is unique up to multiplication with
$e^{\ell 2\pi i/{n+1}}$ for some $\ell\leq  n$. 
There are $n+1$ horizontal separatrices 
at $p$ for the flat metric defined by $\omega$ whose orientations
point towards $p$. For small $u>0$
cut the surface $S$ open along the initial subsegments of length $2u$ 
of these 
$n+1$ horizontal segments.  
The result is a $2n+2$-gon which we refold as in the 
case of a simple zero.

Now let $C$ be a smooth simple loop enclosing the zero
$p\in Z(\omega)$. Then 
the Schiffer variation at $p$ 
is the real part $\delta(C,-1/\omega)$ of 
a complex twist deformation $\delta(C,v)$ of $X$
about $C$ 
where $v$ is a holomorphic
vector field along $C$.

Namely, there is a 
vector ${\rm tw}(C)$ tangent to 
${\cal A\cal P}({\cal Q})$ at $\omega$ defined by 
\[\langle {\rm tw}(C),E\rangle =C\cdot E\]
where $E\in H_1(X,Z(\omega))$ and where $\cdot$ 
is the natural intersection pairing
\[H_1(X-Z(\omega))\times H_1(X,Z(\omega))\to \mathbb{Z}.\]
The tangent space at $\omega$ to the absolute period foliation
is generated by the transformations ${\rm tw}(C_p)$,
$p\in Z(\omega)$, subject to the relation
$\sum {\rm tw}(C_{p_i})=0$ (see p.1236 of \cite{McM13}).
If ${\cal Q}$ is the principal stratum then 
the leaves of
the absolute period foliation ${\cal A\cal P}({\cal Q})$ 
have complex dimension $2g-3$.
We refer to \cite{McM13} for details and an explanation of these
notations.

Note 
that the tangent bundle of ${\cal A\cal P}({\cal Q})$
is naturally equipped with a complex structure $J$ as well as with a 
real structure. The subbundle of $T{\cal A\cal P}({\cal Q})$ spanned by
the twist deformations corresponding to 
the Schiffer variations is a maximal real subbundle for this
real structure. By abuse of notation, we call a twist
deformation corresponding to a Schiffer variation again
a Schiffer variation, i.e. we view Schiffer variations
as tangent vectors of the absolute period foliation.

\begin{example}\label{node}
Let $\omega_1,\omega_2$ be two abelian differentials on 
two closed surfaces $S_1,S_2$ of genus
$g_1,g_2$. Assume that the area of $\omega_i$ is $a_i$ for some
$a_i>0$ with $a_1+a_2=1$. 
Cut a small horizontal slit into
$S_1,S_2$ of the same length. The differentials $\omega_i$ 
define an orientation of these slits. 
Glue $S_1$ to $S_2$ with an orientation
reversing isometry along the slits. 
The result is an area one abelian differential $\omega$ on a surface
of genus $g_1+g_2$ with two singular points $p_1,p_2$ which are connected
by two homologous horizontal saddle connections of the same length.
Assume that the orientation defined by $\omega$ 
of these saddle connections points from $p_1$ to $p_2$.
The deformation induced by the 
Schiffer variation corresponding to the parameters $(-1,1)$
decreases the length of the slit and limits in a surface with
nodes. This surface
with nodes  consists of the surfaces $S_1,S_2$ attached at one point, 
equipped with an abelian differential which maps to 
the differentials $\omega_1,\omega_2$ by the marked point forgetful map.
\end{example}

Let again ${\cal Q}$ be a component of a stratum with 
$k\geq 2$ zeros and let $\hat {\cal Q}$ be a finite 
normal cover of ${\cal Q}$ such that 
there is a consistent numbering of the zeros of $q\in \hat {\cal Q}$ 
varying continuously with $q$. 
We may assume that ${\cal Q}$ is the quotient of 
$\hat{\cal Q}$ by the action of 
a subgroup of the symmetric
group in $k$ variables.

Let $\mathfrak{a}=(a_1,\dots,a_{k})\in \mathbb{R}^{k}$  be any 
$k$-tuple of \emph{real} numbers with $\sum_ia_i=0$.  Then 
$\mathfrak{a}$ defines a smooth vector field $X_{\mathfrak{a}}$ on 
$\hat{\cal Q}$ as follows. For each $\omega\in \hat{\cal Q}$, the value 
of $X_{\mathfrak{a}}$ at $\omega$ is the Schiffer variation for the parameters
$(a_1,\dots,a_{k})$ at the numbered zeros of $\omega$.
Thus $X_{\mathfrak{a}}$ is tangent to the absolute period foliation.

The \emph{Teichm\"uller flow} 
$\Phi^t$ on ${\cal Q}$ lifts to a smooth flow on $\hat {\cal Q}$
denoted by the same symbol. Its 
derivative acts on 
the tangent bundle of $\hat {\cal Q}$. We have

\begin{lemma}\label{invariance2}
$d\Phi^tX_{\mathfrak{a}}=e^{t}X_{\mathfrak{a}}$.
\end{lemma}
\begin{proof} Let $\omega\in \hat {\cal Q}$; then the horizontal
foliation ${\cal F}$ of $\omega$ is defined by
$\omega(T{\cal F})\in \mathbb{R}$. 
The Teichm\"uller flow expands the 
horizontal foliation ${\cal F}$ of 
$\omega$ with the expansion rate $e^{t/2}$. 
Thus if $A_s$ is an arc of length $4\sqrt{s}$
in the imaginary axis for the preferred coordinate near $p$ 
 (recall that
this arc is horizontal for the flat metric defined
by $\omega$)
then the image of $A_s$ in $\Phi^t\omega$ 
is an arc of length $4e^{t/2}\sqrt{s}=
4\sqrt{e^{t}s}$ in the imaginary
axis of a preferred coordinate. Taking derivatives shows the claim.
\end{proof}

The vector field $X_{\mathfrak{a}}$ defines a flow 
$\Lambda_{\mathfrak{a}}^t$ on 
$\hat {\cal Q}$. This flow is incomplete as a horizontal saddle
connection may give rise to a finite flow line limiting on a lower
dimensional stratum (see p.1235 of \cite{McM13}). 
There may also be limit points on 
surfaces with nodes as described in Example \ref{node}.
However, if $q$ does not have any horizontal saddle connection
then the flow line of $\Lambda^t_{\mathfrak{a}}$ through $q$
is defined for all times \cite{MinW14}.
As almost every point with respect to the Lebesgue measure 
$\lambda$ is a differential without horizontal saddle connection,
$\Lambda_{\mathfrak{a}}^t$ defines a flow 
on a subset of $\hat{\cal Q}$ of full Lebesgue measure.

For $\mathfrak{a}\not=\mathfrak{b}$ the 
flows $\Lambda_{\mathfrak{a}}^t$ and $\Lambda_{\mathfrak{b}}^s$
commute and hence these flows  fit together to 
a (local) action of the
group $\mathbb{R}^{k-1}$ on $\hat{\cal Q}$.   
This action is smooth, and its local
orbits naturally develop to a smooth foliation
of $\hat{\cal Q}$ called the \emph{real REL foliation} \cite{MinW14}.
This foliation is a subfoliation of the absolute period foliation.

A leaf of the \emph{local strong unstable foliation} of 
$\hat {\cal Q}$ consists of abelian differentials with the same
horizontal foliation (up to Whitehead moves).

\begin{lemma}\label{stable}
The real REL foliation is a subfoliation 
of the strong unstable foliation of $\hat {\cal Q}$ which is invariant
under the action of the Teichm\"uller flow and under 
holonomy along the strong stable foliation.
\end{lemma}
\begin{proof}
By construction, 
the Schiffer variation defined by the vector field
$X_{\mathfrak{a}}$ preserves
the horizontal foliation of an abelian differential up to Whietehead moves.
Hence the vector fields $X_{\mathfrak{a}}$ are tangent to the
strong unstable foliation of $\hat{\cal Q}$. 
As a consequence, the real REL foliation is a subfoliation of 
the strong unstable foliation, and it is smooth.
We refer to \cite{McM14} for a detailed analysis 
of this foliation in the case $g=2$.

Together with Lemma \ref{invariance2}, this implies
invariance under the action of the Teichm\"uller flow.
Invariance under holonomy along the strong stable foliation 
follows from the fact that a leaf of the real REL foliation
can be characterized as the set of all abelian differentials
in the stratum with fixed horizontal foliation and the property
that the vertical foliations all define the same
absolute cohomology class (see \cite{MinW14}). 
The lemma follows.
\end{proof}

\begin{remark}
It is easy to see that the flows $\Lambda_{\mathfrak{a}}^t$ preserve
the Lebesgue measure $\lambda$ of the stratum. It is an
interesting question whether any of these flows is ergodic. 
Our proof of ergodicity of the absolute period foliation
does not give any information to this end.
\end{remark}


The above discussion shows that the absolute 
period foliation has an affine structure 
(see \cite{MinW14} and p.1236 of \cite{McM13} for
more details and compare also \cite{LNW15}).

%
%

Recall that there is a natural circle action on ${\cal Q}$. 
To a point $e^{is}$ on the unit circle $S^1\subset \mathbb{C}^*$ 
and a quadratic differential 
$q$ we associate the differential $e^{is}q$. 
For $\mathfrak{a}\in \mathbb{R}^{k}$ with zero mean, for 
$e^{is}\in S^1$ and for 
$\omega\in {\cal Q}$ 
let 
\[\Lambda_{e^{is}\mathfrak{a}}^t(\omega)=
e^{-is}\Lambda_{\mathfrak{a}}^t (e^{is}\omega).\]
Then $(t,\omega)\to \Lambda^t_{e^{is}\mathfrak{a}}\omega$
defines a flow on $\hat{\cal Q}$ which 
preserves the absolute period foliation.

Following \cite{EMZ03} we define the \emph{principal boundary}
of the component ${\cal Q}$ of a stratum as follows. 
Let $\omega\in {\cal Q}$ and assume that $\omega$ has a 
horizontal saddle connection and that the set of horizontal
saddle connections of $\omega$ is a forest, i.e. it does not have
cycles. Let $p_1,p_2$ be the endpoints of such a saddle connection
$\alpha$, chosen such that $\alpha$ points from $p_1$ to $p_2$
with respect to the orientation defined by $\omega$,
and let $\mathfrak{a}$ be the vector $(-1,1)$ at $p_1,p_2$. 
(Strictly speaking, this is only defined in $\hat {\cal Q}$, but the 
choice of $\alpha$ singles out two zeros of $\omega$ and hence
this makes also sense in ${\cal Q}$).
Then the arc $t\to \Lambda^t_{\mathfrak{a}}\omega$ limits on a 
differential $\zeta$ for which the points $p_1,p_2$ coalesce, and there are
no other identifications of zeros. 
The differential $\zeta$ 
is contained in a component of a stratum in the boundary of 
${\cal Q}$, and we call such a component a \emph{finite core face}
of the principal boundary of ${\cal Q}$. 
The dimension of a finite core face of ${\cal Q}$ equals 
${\rm dim}({\cal Q})-1$.

A second degeneration which gives rise
to a point in the principal boundary is the contraction
of two or more homologous saddle connections. In this
case the resulting surface is a surface consisting of two
or more smooth components which are connected at nodes.
The sum of the genera of these surfaces equals $g$, and the
resulting abelian differential has a regular point or
a zero at a node. We call a component of abelian differentials 
on surfaces with nodes arising in this way an \emph{infinite face}.
We call the infinite face a \emph{core face} if 
it consists of surfaces comprised of two components which 
are attached at a single separating node.
Note that there are up to  
$\lfloor g/2\rfloor$ core faces which correspond to the \emph{type} of 
the decomposition,
i.e. to a decomposition $g=g_1+g_2$ with $g_1,g_2\geq 1$.
The node defines a marked point on each of the components of the
surface with nodes. We call a point in a core face a \emph{regular} 
point if the node is not a zero of the abelian differential on a 
component surface. A point which is not regular is called \emph{singular}.
The set of singular points is of codimension one.


To summarize, there is a decomposition of the 
principal boundary of ${\cal Q}$ into
\emph{faces} (see p.76 of \cite{EMZ03}). Each face 
either is a component of a stratum in the adherence of 
${\cal Q}$ with fewer zeros, or it corresponds
to a \emph{configuration} which consists of a decomposition
of the surface into surfaces $S_i$ of genus $g_i$ with
$\sum_ig_i=g$, a combinatorial configuration of 
attachment data which organizes the glueing at the nodes
and numbers $a_j\geq 0$ which describe the order of the
zero of the differentials at the node (see \cite{EMZ03} for more
details). If we denote by $\overline{\cal Q}$ the union of 
${\cal Q}$ with its principal boundary, then the core faces 
are the faces of codimension one in $\overline{\cal Q}$.

For the remainder of this section we assume that 
${\cal Q}$ is the principal stratum.
The following structure theorem is 
Lemma 9.8 of \cite{EMZ03}. For its formulation,
note that $\overline{\cal Q}$ properly 
contains the 
entire moduli space of area one abelian differentials.
For $\epsilon >0$ let
$B(\epsilon)$ be the disk of radius $\epsilon$ in the complex
plane.

\begin{proposition}\label{coreface}
Let ${\cal F}$ be an infinite
core face of the principal boundary and 
let $\omega\in {\cal F}$ be a regular point. Then there is a 
number $\epsilon>0$, and there is a 
neighborhood $V$ of $\omega$ in ${\cal F}$,  
a neighborhood 
$U$ of $\omega$ in $\overline{\cal Q}$, 
and a homeomorphism
$\phi:V\times B(\epsilon)\to U$ with the following properties.
\begin{enumerate}
\item $\phi(x,0)=x$ for all $x\in V$.
\item $\phi(\{x\}\times B(\epsilon))\subset {\cal A\cal P}(x)$.
\end{enumerate}
\end{proposition}
\begin{proof} A regular 
point $z\in {\cal F}$ is defined by two abelian differentials
$\omega_1,\omega_2$ on surfaces $S_1,S_2$ 
attached at a marked point $p_1,p_2$. The marked point $p_i$ 
is a regular point for $\omega_i$. 
Take a vector 
$\gamma$ in the complex plane of sufficiently small length
$r<\epsilon$, slit $S_1,S_2$ open
at the marked points in 
direction of $\gamma$ and glue the abelian differentials
$\omega_1,\omega_2$ along the slits.
The resulting differential $\phi(z,\gamma)$ 
has the same absolute periods
as $z$. 

By Lemma 9.8 of \cite{EMZ03}, for a sufficiently small neighborhood
$V$ of $z$ and sufficiently small $\epsilon$ the map 
$\phi$ defines a homeomorphism of $V\times B(\epsilon)$ onto
a neighborhood $U$ of $z$ in $\overline{\cal Q}$ with the properties
stated in the lemma.
\end{proof}

The measure in the statement of the
following lemma is the Lebesgue measure.

\begin{lemma}\label{boundaryintersect}
For almost every $\omega\in {\cal Q}$, the leaf
${\cal A\cal P}(\omega)$ intersects every 
infinite core face of the principal
boundary of ${\cal Q}$ in regular points.
\end{lemma}
\begin{proof} 
We say that a translation surface $\omega$ has an  
\emph{isolated bigon} of type $(g_1,g_2)$ where 
$g_1+g_2=g$ if it admits a pair of homologous saddle connections
$\alpha_1,\alpha_2$ 
connecting two zeros $p_1,p_2$ with the following property.
There is no other saddle connection parallel to 
$\alpha_i$, moreover $\alpha_1\cup\alpha_2$ decomposes
$S$ into a surface of genus $g_1$ and a surface of genus $g_2$.
Since the Teichm\"uller flow preserves saddle connections and 
only changes their direction, the set of points 
$q\in {\cal Q}$ which admit an isolated bigon of 
type $(g_1,g_2)$ is invariant under the Teichm\"uller flow.

The set of directions of a translation surface containing
a saddle connection is countable and hence of measure zero.
Thus Proposition \ref{coreface} shows that for all $g_1,g_2\geq 1$ with 
$g_1+g_2=g$, the set of all 
points $\omega\in {\cal Q}$ which admit an isolated bigon of 
type $(g_1,g_2)$ has positive Lebesgue measure
(see also \cite{EMZ03} for details).
By invariance and 
by ergodicity of the
Teichm\"uller flow on ${\cal Q}$, 
we conclude that this set has full measure.

Let $\omega\in {\cal Q}$ and assume that there is some  
$e^{is}\in S^1$ with the property that 
$e^{is}\omega$ has an isolated horizontal bigon. 
Assume that this bigon is defined by a pair of homologous
saddle connections with endpoints at the zeros 
$p_1,p_2$ of $\omega$. We assume that the points $p_1,p_2$
are ordered in such a way that the saddle connections 
connect $p_1$ to $p_2$ with respect to the
orientation defined by $\omega$.  
Then the flow line $t\to \Lambda_{\mathfrak{a}}^t(e^{is}\omega)$ 
defined by the vector $\mathfrak{a}$ with
coordinates $(-1,1)$ at the points $p_1,p_2$ 
collapses the pair of horizontal saddle connections of $e^{is}\omega$ 
to a point. This means that there is some $\tau>0$ such that
$\Lambda_{\mathfrak{a}}^t(e^{is}\omega)$ is defined for $0\leq t<\tau$, 
and the surfaces $\Lambda^t_{\mathfrak{a}}(e^{is}\omega)$
converge as $t\nearrow \tau$ 
to a surface in the infinite core face of the 
principal boundary of type $(g_1,g_2)$. 
The lemma follows.
\end{proof}

A subset of ${\cal Q}$ is \emph{saturated for the absolute period foliation}
if it is a union of leaves.

\begin{corollary}\label{saturated}
The set ${\cal S}$ of points $\omega\in {\cal Q}$ 
such that ${\cal A\cal P}(\omega)$ intersects
every infinite core face of the principal boundary of ${\cal Q}$ 
is saturated for the absolute period foliation and 
of full Lebesgue measure.
\end{corollary}

A smooth foliation of ${\cal Q}$ is \emph{ergodic} for the
Lebesgue measure if every Borel set $A\subset {\cal Q}$ which 
is saturated for the foliation 
either has full measure or measure zero.

A finite core face of the principal boundary of a
stratum is a component of a stratum. Hence if $g\geq 3$ and if 
this stratum has more than
one zero, then the absolute period foliation is defined. 
For an infinite core face of the principal boundary, the absolute period
foliation is defined as well. Namely, 
such an infinite core face is determined 
by  closed surfaces $S_1,S_2$ of genus
$g_1\geq 1,g_2\geq 1$ and $g_1+g_2=g$. 
Write $S_1\sqcup S_2$ for the surface obtained by attaching
$S_1$ and $S_2$ at a single point, viewed as a surface with a node.
The moduli spaces of abelian differentials on 
$S_1,S_2$ determine a moduli space of abelian
differentials on $S_1\sqcup S_2$ and an absolute period foliation.
We require that the area of a 
differential on $S_1\sqcup S_2$ equals one and
hence the areas of $S_1$ and $S_2$ add up to one. 
In particular, the principal stratum in the 
moduli space of abelian differentials 
on $S_1\sqcup S_2$ decomposes as
$\cup_{a\in (0,1)} 
{\cal Q}_{S_1\sqcup S_2}(a,1-a)$ 
where a differential in ${\cal Q}_{S_1\sqcup S_2}(a,1-a)$ gives area 
$a$ to $S_1$.
Note that for each $a\in (0,1)$, ${\cal Q}_{S_1\sqcup S_2}(a,1-a)$ is a real 
hypersurface in the moduli space of all area one
abelian differentials on $S_1\sqcup S_2$.

\begin{lemma}\label{ergodicone}
If the absolute period foliation of the principal stratum
of $S_1,S_2$ is ergodic then so is the absolute period foliation
of $ {\cal Q}_{S_1\sqcup S_2}(a,1-a)$. 
\end{lemma}
\begin{proof} Write $a_1=a$ and $a_2=1-a$. Then 
${\cal Q}_{S_1\sqcup S_2}(a_1,a_2)$ 
is the space of pairs 
$((\omega_1,p_1),(\omega_2,p_2))$ 
where $\omega_i$ is an abelian differential on $S_i$ of area $a_i$ 
with simple zeros  
and a marked point $p_i$ (here the node
is at the marked point). 

Let 
${\cal Q}_{S_i}(a_i)$ be the principal stratum in the
moduli space of abelian differentials
on $S_i$ of area $a_i$. There is a natural node forgetting map
\[P:{\cal Q}_{S_1\sqcup S_2}(a_1,a_2)\to 
{\cal Q}_{S_1}(a_1)\times {\cal Q}_{S_2}(a_2).\]
This map is a fibration whose 
fibre over a point $(\omega_1,\omega_2)$ 
can naturally be identified with the product 
$(S_1,\omega_1)\times (S_2,\omega_2)$ (it consists of the pair of 
marked points) 
and hence it is equipped with 
a natural Lebesgue measure (the product of the Lebesgue measures 
defined by the differentials $\omega_i$ on the surfaces $S_i$).
The fibration respects absolute periods and therefore 
if $\omega\in {\cal Q}_{S_1\sqcup S_2}(a_1,a_2)$ then $P^{-1}(P\omega)\in 
{\cal A\cal P}(\omega)$.

The Lebesgue measure on
$ {\cal Q}_{S_1\sqcup S_2}(a_1,a_2)$ can locally be described as 
a product of the Lebesgue measure
on the fibre and the Lebesgue measure on 
the base (see \cite{EMZ03} for details).
As a consequence, a Borel set 
$A\subset {\cal Q}_{S_1\sqcup S_2}(a_1,a_2)$ 
which is saturated for the absolute period foliation
maps to a Borel subset of ${\cal Q}_{S_1}(a_1)\times {\cal Q}_{S_2}(a_2)$ which is
saturated for the absolute period 
foliation, and it coincides with
$P^{-1}(PA)$ up to a set of measure zero. Thus ergodicity of the
absolute period 
foliation on ${\cal Q}_{S_i}(a_i)$ implies ergodicity of the
absolute period foliation on $ {\cal Q}_{S_1\sqcup S_2}(a_1,a_2)$.
\end{proof}

As an immediate consequence we obtain

\begin{corollary}\label{structure}
Let $A\subset {\cal Q}$ be a Borel set which is saturated for the
absolute period foliation and of positive Lebesgue measure.
Let $U=V\times B(\epsilon)$ 
be a standard neighborhood in 
an infinite core face ${\cal F}$ defined by a surface with nodes
$S_1\sqcup S_2$. If the absolute period foliation of the principal 
stratum of 
$S_1,S_2$ is ergodic then there is a Borel set 
$C\subset (0,1)$ of positive Lebesgue measure such that
up to a set of measure zero we have
\[A\cap U=(\cup_{s\in C}{\cal Q}_{S_1\sqcup S_2}(s,1-s)\cap V)\times B(\epsilon).\]
\end{corollary}
\begin{proof} By Proposition \ref{coreface}, up to a set of 
measure zero the set $A$ intersects $U$ in a set of the form
$(Z\cap V)\times B(\epsilon)$ where $Z\subset {\cal F}$ is 
a Borel set which is saturated for the absolute period foliation.
If the absolute period foliation of $S_1,S_2$ is ergodic then
Lemma \ref{ergodicone} shows that there is a Borel set
$C\subset (0,1)$ of positive measure such that
$Z=\cup_{s\in C} {\cal Q}_{S_1\sqcup S_2}(s,1-s)$ as claimed. 
\end{proof}

We use Lemma \ref{ergodicone} to show

\begin{proposition}\label{relergodic}
The absolute period foliation ${\cal A\cal P}({\cal Q})$ 
of the principal stratum in genus $g\geq 2$ is ergodic.
\end{proposition}
\begin{proof} We use induction on the genus $g$ of $S$.
The case $g=2,3$ is due to McMullen \cite{McM14}.
Let $g\geq 6$ and assume that the proposition holds true for
$g-4$ and for $g-2$.

Let ${\cal Q}$ be the principal stratum of abelian differentials
on a surface of genus $g$. 
Let $A\subset {\cal Q}$ be a Borel subset which 
is saturated for the absolute period foliation and which is
of positive Lebesgue measure. Then the same holds true for 
$A\cap {\cal S}$ where ${\cal S}\subset {\cal Q}$ is as
in Corollary \ref{saturated}.

Let $S_1,S_3$ be a surface of genus two and let 
$S_1\sqcup S_2\sqcup S_3$ be the surface with two nodes obtained
by attaching $S_1,S_3$ to a surface $S_2$ of genus $g-4$ at a single point each.
Lift ${\cal Q}$ to a finite cover $\hat{\cal Q}$ so that
this configuration determines a subset of $\hat{\cal Q}$ as follows.
Let ${\cal Z}$ be the space of all area one abelian differentials 
on $S_1\sqcup S_2\sqcup S_3$. Note that 
\[{\cal Z}=\cup_{a_i>0,a_1+a_2+a_3=1}{\cal Z}(a_1,a_2,a_3)\]
where an abelian differential in the space
${\cal Z}(a_1,a_2,a_3)$ gives area $a_i$ to $S_i$.
We require that 
there is an open subset $U$ of 
$\overline{\hat{\cal Q}} $ of the form
\[U=V\times B(\epsilon)\times B(\epsilon)\]
with $V\subset {\cal Z}$ open and such that for each 
$x\in V$ the set $\{x\}\times B(\epsilon)\times B(\epsilon)$
is contained in a leaf of the absolute period foliation.
The existence of such an open set $U$ follows from two 
applications of 
Proposition \ref{coreface} and the requirement that
we can distinguish the two surfaces $S_1,S_3$ in the principal
boundary of $\hat{\cal Q}$ (i.e. they are not identified by
an element of the mapping class group).

As before, we conclude that 
since $A$ is saturated for the absolute period foliation,
there is a Borel subset ${\cal Z}_A$ of ${\cal Z}$ which is saturated for the
absolute period foliation and such that up to a set
of measure zero, 
the intersection of $A$ with $V\times B(\epsilon)\times B(\epsilon)$
equals 
\[(V\cap {\cal Z}_A)\times B(\epsilon)\times B(\epsilon)
\cap \hat{\cal Q}.\]

By induction hypothesis, the absolute period foliation on 
$S_i$ is ergodic. Note that as we assume that $g\geq 6$, the 
genus of $S_2$ is at least two.
By Lemma \ref{ergodicone}, this implies that
there is a Borel subset $D_0$ of the set
\[D=\{(a_1,a_2,a_3)\mid a_i>0,a_1+a_2+a_3=1\}\]
such that up to a set of measure zero, we have
${\cal Z}_A=\cup_{x\in D_0}\hat{\cal Q}_{S_1\sqcup S_2\sqcup S_3}(x)$.
Moreover, the Lebesgue measure of $D_0$ is positive.

Let $\Sigma$ be a surface of genus $g-2$
(which should be viewed as
a component of a surface with a single node in the 
Deligne Mumford compactification of the moduli space of $S$
whose second component is the surface $S_1$).
For each 
$a<1$, the surface with nodes 
$S_2\sqcup S_3$ determines an infinite core face ${\cal G}(1-a)$ 
of the moduli space ${\cal H}_\Sigma(1-a)$ 
of abelian differentials on $\Sigma$ of area $1-a$.
By Corollary \ref{structure} and the induction
hypothesis, applied to $S_1$ and $\Sigma$,
there is a Borel set $C_1\subset (0,1)$ such that 
\[A\cap U=\bigl(\cup_{a\in C_1}
(\hat {\cal Q}_{S_1\sqcup \Sigma}(a,1-a)\cap 
V\times B(\epsilon))\times B(\epsilon)\bigr)\cap \hat{\cal Q}.\]
As a consequence, the set $D_0$ is of the form
\begin{equation}\label{right}
D_0=\{(a_1,a_2,a_3)\mid a_1\in C_1,a_1+a_2+a_3=1\}.
\end{equation}
Exchanging the roles of $S_1$ and $S_3$ shows that on the other
hand, there is a Borel set $C_3\subset (0,1)$ such that
\begin{equation}\label{left}
B_0=\{(a_1,a_2,a_3)\mid a_3\in C_3,a_1+a_2+a_3=1\}.
\end{equation}

Let $\lambda$ be the Lebesgue measure on $(0,1)$.
Define ${\rm essup}(C_i)=\sup \{a>0\mid 
\lambda(C_i\cap [a,1])>0\}\in (0,1]$ 
and ${\rm essinf}(C_i)=\inf\{a>0\mid 
\lambda(C_i\cap [0,a])>0\}\in [0,1)$
$(i=1,3)$. It follows from 
equation (\ref{right}) and equation (\ref{left}) that 
$\lambda$-almost every $b<1-{\rm essinf}(C_1)$ is contained in 
$C_3$ (since on the one hand, we can make $a_2$
as small as we wish, on the other hand, 
for fixed $a\in C_1$ we can make $a_2$ as close
to $1-a$ as we wish).
As a consequence, the set $C_3$ is of the form 
$(0,c)$ for a number $c>0$, in particular 
we have ${\rm essinf}(C_3)=0$. 
By symmetry, we conclude that ${\rm essinf}(C_1)=0$ as well and 
hence by the beginning of this paragraph, $C_3=(0,1)=C_1$.
As a consequence, $D_0=D$ which shows
that $A$ is of full Lebesgue measure.
Ergodicity of the absolute period foliation is an immediate
consequence.

The proposition follows by induction if we can 
show ergodicity of the absolute period foliation for $g=4$ and $g=5$. 

We begin with the case $g=4$.
To this end consider a core face of the principal boundary defined by a surface
with nodes $S_1\sqcup S_2$ where $S_i$ is a surface of genus $2$.

Let $A\subset {\cal Q}$ be a Borel set saturated for the absolute period 
foliation and of positive
Lebesgue measure. We use a neighborhood in standard form of a core face of the
principal boundary to deduce that there is a Borel set $C_1\subset (0,1)$ so that
\[A\cap U=(\cup_{a\in C_1}\hat {\cal Q}_{S_1\sqcup S_2}(a,1-a)\cap V)\times B(\epsilon).\]

The principal boundary of the
surface $S_1$ has an infinite core face
which consists in surfaces with a separating node given by
a decomposition of $S_1$ into two tori $T_1,T_2$. The node is a regular point on each
torus. The absolute period foliation is the 
foliation defined by moving the marked point
over the torus. Thus as before, for a fixed number $s>0$, 
there is a forgeful projection 
$P_{\cal T}$ of the moduli space ${\cal T}(s)$ of tori of area 
$s$ with a marked point onto the moduli space of tori
of area $s$ without marked point. 
A Borel set in ${\cal T}(s)$ which is saturated
for the absolute period foliation is the preimage 
under $P_{\cal T}$ of a Borel set $B(s)$ of 
the moduli space of tori of area $s$. 

As a consequence, if we denote again by 
\[P:\cup_{s\in (0,1)}\cup_s{\cal Q}_{T_1\sqcup \Sigma}(s,1-s)\to 
{\cal Q}_{T_1}(s)\times {\cal Q}_\Sigma(1-s)\]
the natural projection where $\Sigma$ is of genul $3$, then 
the set $A$ intersects
a neighborhood of the core face defined by $T_1\sqcup \Sigma$ 
in a set of the form
\[\cup_sP^{-1}B(s)\times {\cal Q}_\Sigma(s)).\]

However, using ergodicity of the absolute period 
foliation for surfaces of genus $2$, we know that
for each $s$, either $B(s)$ is of full measure or of measure zero. 
By this observation, we
can use the above argument in the case $g=4$ as well.

The proof for $g=5$ is completely analogous and will be omitted.
\end{proof}

\section{The principal boundary of  affine invariant manifolds}

The goal of this section is to study 
the intersection of an \emph{affine
invariant submanifold} of a stratum ${\cal Q}$ with $k\geq 2$ zeros
with a leaf of the absolute period foliation. 

To this end call a connected subset $B$ of a 
leaf of ${\cal A\cal P}({\cal Q})$ \emph{complex affine}
if each point $p\in B$ has an open neighborhood $U$ which 
in affine coordinates is an open subset of a complex affine subspace. 
This is equivalent to stating that $B$ is a smooth submanifold of 
a leaf of ${\cal A\cal P}({\cal Q})$ 
whose tangent bundle $TB$ is invariant under the
complex structure and the real structure and whose 
lift to $\hat{\cal Q}$ is 
invariant under all flows $\Lambda^t_{e^{is}\mathfrak{a}}$ 
whenever $e^{is}X_{\mathfrak{a}}\in TB$.
Here as before, $\hat{\cal Q}$ is a finite cover of 
${\cal Q}$ so that the zeros of differentials in $\hat{\cal Q}$ 
are numbered. Moreover, for any vector $\mathfrak{a}\in \mathbb{R}^k$
of zero mean, $X_{\mathfrak{a}}$ is the Schiffer variation defined by
$\mathfrak{a}$.

The \emph{rank} of an affine invariant manifold ${\cal C}$ is defined by
\[{\rm rk}({\cal C})=\frac{1}{2}{\rm dim}(pT{\cal C})\]
where $p$ is the projection of period coordinates into absolute cohomology.

\begin{lemma}\label{affinelemma}
Let ${\cal C}\subset {\cal Q}$ be an affine invariant submanifold.
Then for every $\omega\in {\cal C}$, the intersection 
${\cal C}\cap {\cal A\cal P}(\omega)$ is a complex affine subspace of 
complex dimension ${\rm dim}({\cal C})-2{\rm rk}({\cal C})$.
\end{lemma}
\begin{proof} A proper affine invariant submanifold of  
${\cal Q}$ lifts to a proper affine invariant submanifold
of $\hat {\cal Q}$, so it suffices to consider
such manifolds ${\cal C}$ in $\hat{\cal Q}$.
Let $r={\rm dim}({\cal C})-2 {\rm rk}({\cal C})$. We may assume that
$r>0$. 
Then for each $q\in {\cal C}$ there is a vector 
$X\in T_q{\cal A\cal P}(\hat{\cal Q})$ which is 
tangent to ${\cal C}$.
By invariance of ${\cal C}$ under
the Teichm\"uller flow, we have
$d\Phi^t(X)\in 
T{\cal A\cal P}(\hat{\cal Q})\cap T{\cal C}$  
for all $t$.

A vector 
$X\in T{\cal A\cal P}(\hat{\cal Q})\cap T{\cal C}$ decomposes as
$X=X^u+X^s$ where $X^u\in T{\cal A\cal P}(\hat{\cal Q})$
is real (and hence tangent to the strong unstable foliation) 
and $X^s$ is imaginary (and hence tangent
to the strong stable foliation). 
We claim that we can find a vector 
$Y\in T{\cal C}\cap T{\cal A\cal P}(\hat{\cal Q})$ 
which either is tangent to the strong
unstable or to the strong stable foliation. 
To this end we may assume that $X^u\not=0$. Since this is an open
condition and since the Teichm\"uller flow on ${\cal C}$
is topologically transitive, we may furthermore assume 
that the $\Phi^t$-orbit 
of the footpoint $q$ of $X$ is 
dense in ${\cal C}$. 
Then there is a sequence
$t_i\to \infty$ such that $\Phi^{t_i}(q)\to q$. 

Choose any smooth norm $\Vert \,\Vert$ on $T\hat{\cal Q}$. 
Up to passing to a subsequence, 
\[d\Phi^{t_i}(X)/\Vert d\Phi^{t_i}(X)\Vert\] converges to a vector
$Y\in T_q{\cal A\cal P}(\hat {\cal Q})$ 
which is tangent to the strong unstable foliation.
As the bundle $T{\cal C}\cap T{\cal A\cal P}(\hat{\cal Q})$ 
is a smooth $d\Phi^t$-invariant subbundle of the 
restriction of the tangent bundle of 
$\hat{\cal Q}$ to ${\cal C}$, we have 
$Y\in T{\cal C}\cap T{\cal A\cal P}(\hat{\cal Q})$ which is what we
wanted to show.

Using Lemma \ref{invariance2} and density of the $\Phi^t$-orbit of $q$,
if $0\not=\mathfrak{a}\in \mathbb{R}^{k}$ 
is such that 
$Y=X_{\mathfrak{a}}(q)$
then $X_{\mathfrak{a}}(u)\in T{\cal C}$
for all $u\in {\cal C}$. 
As a consequence,
${\cal C}$ is invariant under the flow $\Lambda_{\mathfrak{a}}^t$.

By invariance of $T{\cal C}$ 
under the complex structure $J$,
if $r=1$ then 
\[T{\cal C}\cap T{\cal A\cal P}(\hat {\cal Q})=
\mathbb{R}X_{\mathfrak{a}}\oplus J\mathbb{R}X_{\mathfrak{a}}\]
and we are done. Otherwise 
there is a tangent vector $X\in T{\cal C}\cap
T{\cal A\cal P}(\hat {\cal Q})-\mathbb{C}X_{\mathfrak{a}}$.
Apply the above argument to $X$, perhaps via
replacing the Teichm\"uller flow by its inverse. In finitely many
such steps we conclude that
there is a smooth subbundle 
${\cal R}$ of $T{\cal C}\cap T{\cal A\cal P}(\hat {\cal Q})$
which is tangent to the strong unstable foliation 
(i.e. real for the real structure) and of rank $r$ such that
$T{\cal C}\cap T{\cal A\cal P}(\hat{\cal Q})=\mathbb{C}{\cal R}$.
Moreover,  
if $\omega\in {\cal C}$ and if  
$X_{\mathfrak{a}}(\omega)\in {\cal R}$ 
then $X_{\mathfrak{a}}(q)\in {\cal R}$ for every $q\in {\cal C}$ and 
${\cal C}$ is invariant under the flow
$\Lambda^t_{\mathfrak{a}}$.
The same argument applied to the imaginary
subbundle $i{\cal R}$ of $T{\cal C}$ 
and equivariance under the
action of the circle group of rotations 
yields the statement of the lemma.
\end{proof}

We now always assume that 
${\cal C}\subset {\cal Q}$ is an affine invariant manifold
such that the (complex) dimension of 
$T{\cal C}\cap T{\cal A\cal P}({\cal Q})$ 
is at least one. We call such an affine invariant
manifold \emph{redundant}. 
By Lemma \ref{affinelemma}, there are Schiffer variations 
$X_{\mathfrak{a}}\subset 
T{\cal A\cal P}(\hat{\cal Q})$ which are tangent
to ${\cal C}$.


Let $q\in {\cal Q}$ and 
let $\alpha$ be a horizontal saddle
connection for $q$ of length $\ell(\alpha)$; here
the length is taken with respect to the flat metric. 
For a vector $\mathfrak{a}\in \mathbb{R}^k$ with zero mean 
define the \emph{oriented $\mathfrak{a}$-weighted length of 
$\alpha$} by
\[w_{\mathfrak{a}}(\alpha)=\ell(\alpha)/b\]
where $b$ is the oriented 
difference of the coordinates of $\mathfrak{a}$ 
at the endpoints of $\alpha$ (i.e. the weight of the incoming 
endpoint minus the weight of the outgoing endpoint for the orientation
defined by $\omega$) 
provided that this difference does not vanish, and define
$w_{\mathfrak{a}}(\alpha)=\infty$ otherwise.
Note that if we replace $\mathfrak{a}$ by any nonzero 
real multiple then the oriented
$\mathfrak{a}$-weighted length of a saddle connection 
multiplies with the inverse of the
multiple. As a consequence, the set of 
horizontal saddle connections
whose oriented $\mathfrak{a}$-weighted length is 
positive and minimal 
is invariant under scaling
$X_{\mathfrak{a}}$ with a positive number.

An \emph{oriented cycle} of horizontal saddle connections 
for an abelian differential $\omega$ is an embedded
closed curve in $S$ 
consisting of at least two horizontal
saddle connections, and these saddle connections
are equipped with the orientation induced by $\omega$. 
We do not require that these orientations fit together to an 
orientation of the simple closed curve. 
An example of an oriented cycle is 
a \emph{bigon} which 
consists of two distinct 
homologous horizontal saddle connections of the same
length which meet at two distinct
endpoints with an angle an integral 
multiple of $2\pi$. 
Cycles of combinatorial length $u$ 
pass through $u$ distinct zeros of $\omega$.

\begin{remark} 
As explained on p.73 of \cite{EMZ03}, a translation surface
which is generic for the Lebesgue measure in a stratum 
has infinitely many bigons of saddle connections- in fact,
the number of these bigons grows in length with the same
rate as the number of all saddle connections.
\end{remark}

\begin{definition}\label{flexible}
Let ${\cal C}\subset \hat {\cal Q}$ be a redundant affine invariant
manifold 
and let $X_{\mathfrak{a}}\in T{\cal A\cal P}(\hat{\cal Q})$ 
be tangent to ${\cal C}$.
Define $q\in {\cal C}$ to be 
\emph{flexible} for $X_\mathfrak{a}$ if there is some $s\in [0,2\pi)$ 
with the following properties.
\begin{enumerate}
\item 
The graph $G$ of horizontal saddle connections for 
$e^{is}q$ is nonempty and contains elements of finite
$\mathfrak{a}$-weighted length. 
\item 
The subgraph of $G$ of saddle connections for $e^{is}q$  
with smallest positive 
$\mathfrak{a}$-weighted length does not contain cycles. 
\end{enumerate}
A translation surface which is not flexible for $X_{\mathfrak{a}}$ is called 
\emph{rigid} for $X_{\mathfrak{a}}$.
We call the vector field 
$X_{\mathfrak{a}}$ \emph{flexible} for ${\cal C}$ if 
there is a flexible translation surface $q\in {\cal C}$ for
$X_{\mathfrak{a}}$.
\end{definition}

Note that if $q\in {\cal C}$ is flexible for 
$X_{\mathfrak{a}}$ then $q$ is flexible for 
$cX_{\mathfrak{a}}$ for every $c\in \mathbb{R}-\{0\}$ 
(to see this for a negative number $c$ use the 
fact that
$-\omega=e^{\pi i}\omega\in {\cal C}$).
In particular, if we denote by
${\cal V}$ the real subspace 
of the complex
vector space $T{\cal C}\cap T{\cal A\cal P}({\cal Q})$ then we
can talk about flexible points in the
projectivization $P{\cal V}\subset P\mathbb{R}^k$ of 
${\cal V}$.

The next lemma states the basic properties of flexible
vector fields.

\begin{lemma}\label{invariantflex}
\begin{enumerate}
\item If $[X_{\mathfrak{a}}]\in P{\cal V}$ is flexible for $q\in {\cal C}$ then
there are open neighborhoods $V$ of $[X_{\mathfrak{a}}]$ in 
$P{\cal V}$, $U$ of $q$ in ${\cal C}$ such that every $[Y]\in V$ is 
flexible for every $z\in U$.
\item For every 
$[X_{\mathfrak{a}}]\in P{\cal V}$ the set of all
$[X_{\mathfrak{a}}]$-rigid points in ${\cal C}$ 
is a finite (perhaps empty) 
union of affine invariant manifolds which is proper
if and only if $[X_{\mathfrak{a}}]$ is flexible for ${\cal C}$.
\end{enumerate}
\end{lemma} 
\begin{proof}  
Let $\mathfrak{a}\in \mathbb{R}^k$ and assume that 
$q\in {\cal C}$ is flexible for $X_{\mathfrak{a}}$. 
Let $e^{is}$ be such that 
the graph $G$ of horizontal saddle connections of $e^{is}q$ of 
minimal positive $\mathfrak{a}$-weighted length 
is not empty and  does not  
have cycles. Recall that $G$ is composed of finitely many
saddle connections. 

Let $\alpha_1,\dots,\alpha_k$ 
be the saddle connections of minimal
positive $\mathfrak{a}$-weighted length. 
Since saddle connections for differentials $q$ 
depend smoothly on $q$ as long as $q$ remains in 
a fixed component of a stratum,
for each $z\in {\cal C}$ which 
is sufficiently close to $e^{is}q$ and for each $i$ there is 
a saddle connection $\alpha_i(z)$ 
for $z$ which is homotopic to 
$\alpha_i$ with fixed endpoints (for the natural identification of 
nearby surfaces) 
and which depends smoothly
on $z$. Moreover, the set of saddle connections
in the direction of $\alpha_i(z)$ of length 
bounded from above by some fixed number
can be obtained by
a smooth deformation of some (perhaps not all)
saddle connections of $q$.

As a consequence, for each point $z$ in 
a neighborhood $W$ of $q$ there is some $i$ such that 
$\alpha_i(z)$ is of minimal positive $\mathfrak{a}$-weighted length
in its direction. Moreover, 
since the graph $G$ has no cycles there are no cycles of saddle
connections of minimal positive $\mathfrak{a}$-weighted length 
on $z\in W$ whose direction coincides with the direction of 
$\alpha_i(z)$. This is what we wanted to show.

To summarize,
the set of points which are flexible for 
$X_{\mathfrak{a}}$ is open. The same reasoning shows that a point which is
flexible for $X_{\mathfrak{a}}$ is flexible for 
$X_{\mathfrak{b}}$ for every $b\in \mathbb{R}$ sufficiently close to
$\mathfrak{a}$. The
first part of the lemma follows.

To show the second part of the lemma, 
by Theorem 2.2 of \cite{EMM13} 
it suffices to show that the set of points in 
${\cal C}$ which are flexible for $X_{\mathfrak{a}}$ 
is invariant under the action of the group $SL(2,\mathbb{R})$.  
Now the group $SL(2,\mathbb{R})$ is generated by the
circle group of rotations and the group of diagonal matrices
and therefore  
it suffices to show that the set of all flexible  
points for $X_{\mathfrak{a}}$ in ${\cal C}$ 
is invariant under the action of 
these two groups.

Invariance under the action of the circle group of rotations
is immediate from the definition. To show invariance under the action 
of the diagonal group note that this action preserves saddle connections
and maps saddle connections with the same slope to
saddle connections with the same slope. 
Moreover, lengths of saddle connections with the same
slope are multiplied with the same constant.
The claim then follows from 
Lemma \ref{invariance2}. 
\end{proof}

\begin{example}\label{decagon}
1) There is an affine invariant manifold ${\cal C}$
of rank one and dimension 3 
in the principal stratum ${\cal H}(1,1)$ 
of the moduli space of abelian differentials 
on a surface of genus $g=2$. This manifold 
is mapped by the composition of the projection
${\cal H}(1,1)\to {\cal M}_2$ (here ${\cal M}_2$ is the 
moduli space of curves of genus 2) with the Torelli map
into the 
Hilbert modular surface ${\bf H}^2\times {\bf H}^2/SL(2,{\cal O}_{\sqrt{5}})$ 
for discriminant $D=5$. 
The principal boundary of ${\cal C}$ intersects the finite
face ${\cal H}(2)$ of ${\cal H}(1,1)$  
in a Teichm\"uller curve. The rigid set for the (unique up to scale)
real vector field $X_{\mathfrak{a}}$ consists 
of another Teichm\"uller curve,
the curve through the translation surface defined by the 
regular decagon. For no other
discriminant such a rigid curve in the corresponding
rank one affine invariant manifold exists \cite{McM05,McM06}.

2) Let ${\cal B}$ be the moduli space
of holomorphic differentials on $\mathbb{C}P^1$ 
with 12 simple poles and 4 double zeros. 
Taking a two-sheeted branched cover with a branch point 
at each of the singular points defines an 
$SL(2,\mathbb{R})$-invariant closed subset ${\cal C}$ of 
the stratum ${\cal Q}$ of abelian differentials with four 
double zeros on a surface of genus $5$. 
The hyperelliptic involution acts on saddle connections
connecting the zeros and hence saddle connections in a given
direction come in pairs. As a consequence, 
${\cal C}$ is contained in the rigid set of a flexible
vector field.
\end{example}

As in the case of strata, we can talk about 
the principal boundary of an affine invariant manifold and
its faces. 
The goal of the following proposition
is to investigate the finite faces of the principal 
boundary of a redundant affine invariant manifold
${\cal C}\subset {\cal Q}$. If $k\geq 2$ is the number
of zeros of ${\cal Q}$ then 
these finite faces are contained in strata whose
number of zeros is smaller than $k-1$.

By compatibility of the $SL(2,\mathbb{R})$-action with 
closures of strata, the closure of an affine invariant
manifold ${\cal C}\subset {\cal Q}$ in the entire 
moduli space ${\cal H}$ of 
abelian differentials is a closed $SL(2,\mathbb{R})$-invariant
set and hence a finite union of affine
invariant manifolds \cite{EMM13}. Note that this closure may
just be ${\cal C}$, e.g. when ${\cal C}$ is a Teichm\"uller curve.

If this closure does \emph{not} coincide with
${\cal C}$ then we can talk about the finite
principal boundary of ${\cal C}$ and its faces as before. 
It is a finite union of affine invariant manifolds.
If ${\cal B}\subset \overline{\cal C}$ is such a
finite face then we say that ${\cal B}\subset \overline{\cal C}$
is \emph{embedded in standard form} if every 
$q\in {\cal B}$ has a neighborhood $V$ with the following
property. There exists a number $\epsilon >0$ and a homeomorphism
$\phi:V\times B(\epsilon)\to \overline{\cal C}$ onto a neighborhood
of $q$ in $\overline{\cal C}$ such that $\phi(z,0)=z$ and 
$\phi(\{z\}\times B(\epsilon))\subset {\cal A\cal P}(z)$ for all 
$z\in V$.
 
 In the next proposition, the number of zeros of the component 
 of $\overline{\cal C}-{\cal C}$ may be strictly smaller than $k-1$.

\begin{proposition}\label{nofullrank}
Let ${\cal C}\subset {\cal Q}$ be an affine invariant
manifold of rank $\ell\geq 1$ and dimension $2\ell+r$ for some
$r>0$ and let $\overline{\cal C}$ be the closure of 
${\cal C}$ in ${\cal H}$. 
If ${\cal C}$ is flexible then 
$\overline{\cal C}-{\cal C}$ contains a nonempty 
finite union of affine invariant
manifolds of rank $\ell$ and dimension $2\ell+r-1$
which are embedded in $\overline{\cal C}$ in
standard form.
\end{proposition}
\begin{proof} As before, we pass to the cover $\hat{\cal Q}$ with
numbered zeros. Thus 
let ${\cal C}\subset \hat {\cal Q}$ be an affine
invariant manifold of rank $\ell\geq 1$ and dimension
$2\ell+r$ for some $r>0$. 
By Lemma \ref{affinelemma},
${\cal C}$ intersects each leaf of the absolute period foliation
in an affine subspace. In particular, there is a linear 
subspace ${\cal V}$ of $\mathbb{R}^k$ of points of zero mean so that
for each $\mathfrak{a}\in {\cal V}$ and 
every $q\in {\cal C}$, the flow line $\Lambda_{\mathfrak{a}}^tq$
is contained in ${\cal C}$ as long as it is defined.

Assume that there is some $q\in {\cal C}$ and 
some $\mathfrak{a}\in {\cal V}$ such that
$q$ is flexible for $X_{\mathfrak{a}}$. 
By definition, there is some $s\in [0,2\pi)$ such that
the differential $e^{is}q$ admits a finite graph $G$ of saddle connections
of minimal positive oriented 
$\mathfrak{a}$-weigthed length, and this graph does 
not contain any cycles. As in the proof of Lemma \ref{invariantflex},
we may assume that this graph 
depends smoothly on $q$ in the following sense.
There is a local smooth transversal ${\cal V}\subset{\cal C}$
to the circle action 
and a smooth function $\sigma:{\cal V}\to \mathbb{R}$ through
$\sigma(q)=s$ so that for all $z\in {\cal V}$ there is a 
graph $G(z)$ of saddle connections on $z$ which are horizontal
for $e^{i\sigma(z)}z$ so that $G(z)$ is of minimal $\mathfrak{a}$-weighted
length and depends smoothly on $z$.

Let $\alpha\subset G(z)$ be such a saddle 
connection of minimal $\mathfrak{a}$-weighted length.
Then the length of $\alpha$ decreases under the Schiffer variation
through $e^{i\sigma(z)}z$ which is 
defined by $X_{\mathfrak{a}}$.  
We claim that the flow $\Lambda_{\mathfrak{a}}^t$ of $X_{\mathfrak{a}}$
through $e^{i\sigma{z}}z$ 
limits on a surface in $\overline{{\cal C}}$ which does not have nodes. 
Namely, by assumption $G(z)$ is a finite union of trees 
and hence by the definition of 
the priented $\mathfrak{a}$-weighted length and the properties of the 
Schiffer variations discussed in Section \ref{absoluteperiod}, 
the limiting surface is obtained by collapsing each of these trees 
to a point and hence identifying the vertices of each of these trees.
Hence if $G(z)$ is connected and has $u\geq 1$ edges 
then the limiting surface is contained 
in the stratum ${\cal Q}_1\subset \overline{\cal Q}$ of differentials with 
$k-u$ zeros. 
 
As $z$ varies through ${\cal V}$, the limiting surfaces define
a subset $V$ of a component of a stratum with $k-u$ zeros
which is contained in the closure of ${\cal C}$. 
This set is transverse to 
the action of the unit circle by 
multiplication with a complex number
of absolute value one. 
The map can be extended to an $S^1$-equivariant map of 
a neighborhood of ${\cal V}$ onto a neighborhood of the image.
The complex
dimension of the image equals the complex dimension 
of ${\cal C}$ minus one.

As a consequence, the closure of ${\cal C}$ intersects 
a boundary stratum of ${\cal Q}$ of codimension at least one in a set
which is of dimension at least ${\rm dim}({\cal C})-1$.
The statement of the proposition now follows from invariance
of the closure of ${\cal C}$ 
under the action of $SL(2,\mathbb{R})$.
\end{proof}

\noindent
{\bf Conjecture:} A rank $\ell $ submanifold of 
a rank $\ell\geq 2$ affine invariant
manifold ${\cal C}$ 
is contained in the rigid set of a 
flexible vector field for ${\cal C}$.

\bigskip\noindent
MATHEMATISCHES INSTITUT DER UNIVERSIT\"AT BONN\\
ENDENICHER ALLEE 60\\ 
53115 BONN, 
GERMANY\\
e-mail: ursula@math.uni-bonn.de


\begin{thebibliography}{McM03a}

 




















\bibitem[CDF15]{CDF15} G.~Calsamiglia, B.~Deroin, S.~Francaviglia,
{\em A transfer principle: From periods to isoperiodic
foliations}, arXiv:1511.07635.




\bibitem[C09]{C09} Y.~Coud\`ene, 
{\em A short proof of unique ergodicity of 
horocyclic flows}, Contemp. Math. 485 (2009), 85--89.

\bibitem[EM01]{EM01} A.~Eskin, H.~Masur, {\em Asymptotic
formulas on flat surfaces}, Ergod. Th. \& Dynam. Sys. 21
(2001), 443--478.


\bibitem[EMZ03]{EMZ03} A.~Eskin, H.~Masur, A.~Zorich,
{\em Moduli spaces of abelian differentials, the principal
boundary, counting problems and the Siegel-Veech constants},
Publ. IHES (2003), 61--179.




\bibitem[EMM13]{EMM13} A.~Eskin, M.~Mirzakhani,
A.~Mohammadi, {\em Isolation theorems 
for $SL(2,\mathbb{R})$-invariant submanifolds in moduli
space}, arXiv:1305.3015, to appear in Ann. Math.
































\bibitem[LNW15]{LNW15} E.~Lanneau, D.-M.~Nguyen, A.~Wright,
{\em Finiteness of Teichm\"uller curves in non-arithmetic
rank 1 orbit closures}, arXiv:1504.03742.









\bibitem[McM05]{McM05} C.~McMullen, {\em Teichm\"uller curves
in genus two: discriminant and spin},
Math. Ann. 333 (2005), 87--130.



\bibitem[McM06]{McM06} C.~McMullen, {\em Teichm\"uller curves
in genus two: torsion divisors and ratios of sines},
Invent. Math. 165 (2006), 651--672.





\bibitem[McM13]{McM13} C.~McMullen, 
{\em Navigating moduli space with complex twists},
J. Eur. Math. Soc. 5 (2013), 1223--1243.


\bibitem[McM14]{McM14} C.~McMullen,
{\em Moduli spaces of isoperiodic forms on 
Riemann surfaces}, Duke Math. J. 163 (2014), 2271--2321.




\bibitem[MinW14]{MinW14} Y.~Minsky, B.~Weiss, 
{\em Cohomology classes represented by measured foliations, 
and Mahler's question for interval exchanges},
Ann. Sci. ENS 47 (2014), 245--284.


\bibitem[MirW15]{MW15} M.~Mirzakhani, A.~Wright,
{\em The boundary of an affine invariant submanifold},
arXiv:1508.01446.


























\end{thebibliography}
\end{document}